\documentclass[11pt,leqno]{article}
\usepackage[margin=1in]{geometry} 
\usepackage{amssymb,amsfonts,amsmath,bbm,mathrsfs,stmaryrd,mathtools}
\usepackage{xcolor}
\usepackage{url}

\usepackage{graphicx}

\usepackage{accents}

\usepackage{extarrows}

\usepackage[shortlabels]{enumitem}
\usepackage{tensor}

\usepackage{xr}
\usepackage[T1]{fontenc}
\usepackage[utf8]{inputenc}

\usepackage[colorlinks,
linkcolor=black!75!red,
citecolor=blue,
pdftitle={},
pdfproducer={pdfLaTeX},
pdfpagemode=None,
bookmarksopen=true,
bookmarksnumbered=true,
backref=page]{hyperref}

\usepackage{tikz}
\usetikzlibrary{arrows,calc,decorations.pathreplacing,decorations.markings,decorations.shapes,intersections,shapes.geometric,through,fit,shapes.symbols,positioning,decorations.pathmorphing}

\makeatletter
\newlength\zig@L
\newlength\zig@La
\newlength\zig@Lb

\newcommand{\xzigrightarrow}[2][]{%
  \mathrel{%
    \settowidth{\zig@La}{$\scriptstyle #2$}%
    \settowidth{\zig@Lb}{$\scriptstyle #1$}%
    \zig@L=\zig@La\relax
    \ifdim\zig@Lb>\zig@L \zig@L=\zig@Lb\fi
    \advance\zig@L by 2.2em\relax
    \tikz[baseline=-0.65ex]{%
      \draw[->,
            line cap=round,
            decorate,
            decoration={zigzag,segment length=4pt,amplitude=1.1pt}]%
        (0,0) -- (\zig@L,0)
        node[midway,above=2pt] {$\scriptstyle #2$}%
        \if\relax\detokenize{#1}\relax\else
          node[midway,below=2pt] {$\scriptstyle #1$}%
        \fi
      ;
    }%
  }%
}
\makeatother

\makeatletter
\newcommand{\squigjoin}{1mu} 

\def\sqleft@{\sim}                    
\def\sqmid@{\sim\mkern-\squigjoin}    

\def\rightsquigarrowfill@{%
  \arrowfill@{\sqleft@}{\sqmid@}{\mkern-4mu\succ}%
}

\newcommand{\xrightsquigarrow}[2][]{%
  \ext@arrow 0359\rightsquigarrowfill@{#1}{#2}%
}
\makeatother

\makeatletter
\newcommand*\circled[1]{\tikz[baseline=(char.base)]{
    \node[shape=circle, draw, inner sep=0pt, 
    minimum height={\f@size},] (char) {\vphantom{WAH1g}#1};}}
\makeatother

\makeatletter
\DeclareRobustCommand\widecheck[1]{{\mathpalette\@widecheck{#1}}}
\def\@widecheck#1#2{%
    \setbox\z@\hbox{\m@th$#1#2$}%
    \setbox\tw@\hbox{\m@th$#1%
       \widehat{%
          \vrule\@width\z@\@height\ht\z@
          \vrule\@height\z@\@width\wd\z@}$}%
    \dp\tw@-\ht\z@
    \@tempdima\ht\z@ \advance\@tempdima2\ht\tw@ \divide\@tempdima\thr@@
    \setbox\tw@\hbox{%
       \raise\@tempdima\hbox{\scalebox{1}[-1]{\lower\@tempdima\box
\tw@}}}%
    {\ooalign{\box\tw@ \cr \box\z@}}}
\makeatother

\usepackage{braket}

\usepackage[amsmath,thmmarks,hyperref]{ntheorem}
\usepackage{cleveref}

\newcommand\nthalias[1]{\AddToHook{env/#1/begin}{\crefalias{lemma}{#1}}}

\nthalias{definition}
\nthalias{example}
\nthalias{examples}
\nthalias{remark}
\nthalias{remarks}
\nthalias{convention}
\nthalias{notation}
\nthalias{construction}
\nthalias{sketch}
\nthalias{theoremN}
\nthalias{propositionN}
\nthalias{corollaryN}
\nthalias{lemma}
\nthalias{proposition}
\nthalias{corollary}
\nthalias{theorem}
\nthalias{conjecture}
\nthalias{question}
\nthalias{assumption}

\creflabelformat{enumi}{#2#1#3}

\crefname{section}{Section}{Sections}
\crefformat{section}{#2Section~#1#3} 
\Crefformat{section}{#2Section~#1#3} 

\crefname{subsection}{\S}{\S\S}
\AtBeginDocument{%
  \crefformat{subsection}{#2\S#1#3}%
  \Crefformat{subsection}{#2\S#1#3}%
}

\crefname{subsubsection}{\S}{\S\S}
\AtBeginDocument{%
  \crefformat{subsubsection}{#2\S#1#3}%
  \Crefformat{subsubsection}{#2\S#1#3}%
}

\theoremstyle{plain}

\newtheorem{lemma}{Lemma}[section]

\newtheorem{corollary}[lemma]{Corollary}
\newtheorem{theorem}[lemma]{Theorem}

\newtheorem{question}[lemma]{Question}

\theoremstyle{plain}
\theoremnumbering{Alph}

\theoremstyle{plain}
\theorembodyfont{\upshape}
\theoremsymbol{\ensuremath{\blacklozenge}}

\newtheorem{remarks}[lemma]{Remarks}

\crefname{definition}{definition}{definitions}
\crefformat{definition}{#2definition~#1#3} 
\Crefformat{definition}{#2Definition~#1#3} 

\crefname{ex}{example}{examples}
\crefformat{example}{#2example~#1#3} 
\Crefformat{example}{#2Example~#1#3} 

\crefname{exs}{example}{examples}
\crefformat{examples}{#2example~#1#3} 
\Crefformat{examples}{#2Example~#1#3} 

\crefname{remark}{remark}{remarks}
\crefformat{remark}{#2remark~#1#3} 
\Crefformat{remark}{#2Remark~#1#3} 

\crefname{remarks}{remark}{remarks}
\crefformat{remarks}{#2remark~#1#3} 
\Crefformat{remarks}{#2Remark~#1#3} 

\crefname{convention}{convention}{conventions}
\crefformat{convention}{#2convention~#1#3} 
\Crefformat{convention}{#2Convention~#1#3} 

\crefname{notation}{notation}{notations}
\crefformat{notation}{#2notation~#1#3} 
\Crefformat{notation}{#2Notation~#1#3} 

\crefname{table}{table}{tables}
\crefformat{table}{#2table~#1#3} 
\Crefformat{table}{#2Table~#1#3}

\crefname{lemma}{lemma}{lemmas}
\crefformat{lemma}{#2lemma~#1#3} 
\Crefformat{lemma}{#2Lemma~#1#3} 

\crefname{proposition}{proposition}{propositions}
\crefformat{proposition}{#2proposition~#1#3} 
\Crefformat{proposition}{#2Proposition~#1#3} 

\crefname{propositionN}{proposition}{propositions}
\crefformat{propositionN}{#2proposition~#1#3} 
\Crefformat{propositionN}{#2Proposition~#1#3} 

\crefname{corollary}{corollary}{corollaries}
\crefformat{corollary}{#2corollary~#1#3} 
\Crefformat{corollary}{#2Corollary~#1#3} 

\crefname{corollaryN}{corollary}{corollaries}
\crefformat{corollaryN}{#2corollary~#1#3} 
\Crefformat{corollaryN}{#2Corollary~#1#3} 

\crefname{theorem}{theorem}{theorems}
\crefformat{theorem}{#2theorem~#1#3} 
\Crefformat{theorem}{#2Theorem~#1#3} 

\crefname{theoremN}{theorem}{theorems}
\crefformat{theoremN}{#2theorem~#1#3} 
\Crefformat{theoremN}{#2Theorem~#1#3} 

\crefname{enumi}{}{}
\crefformat{enumi}{#2#1#3}
\Crefformat{enumi}{#2#1#3}

\crefname{assumption}{assumption}{Assumptions}
\crefformat{assumption}{#2assumption~#1#3} 
\Crefformat{assumption}{#2Assumption~#1#3} 

\crefname{construction}{construction}{Constructions}
\crefformat{construction}{#2construction~#1#3} 
\Crefformat{construction}{#2Construction~#1#3} 

\crefname{sketch}{sketch}{Sketches}
\crefformat{sketch}{#2sketch~#1#3} 
\Crefformat{sketch}{#2Sketch~#1#3} 

\crefname{question}{question}{Questions}
\crefformat{question}{#2question~#1#3} 
\Crefformat{question}{#2Question~#1#3} 

\crefname{equation}{}{}
\crefformat{equation}{(#2#1#3)} 
\Crefformat{equation}{(#2#1#3)}

\numberwithin{equation}{section}

\theoremstyle{nonumberplain}
\theoremsymbol{\ensuremath{\blacksquare}}

\newtheorem{proof}{Proof}
\newcommand\pf[1]{\newtheorem{#1}{Proof of \Cref{#1}}}

\newcommand\cM{{\mathcal M}}

\DeclareMathOperator{\id}{id}

\newcommand{\comment}[1]{}

\newcommand{\xrightarrowdbl}[2][]{%
  \xrightarrow[#1]{#2}\mathrel{\mkern-14mu}\rightarrow
}

\title{Automatic cocommutativity for the quantum-group comonad}
\author{Alexandru Chirvasitu}

\begin{document}

\date{}

\newcommand{\Addresses}{{
  \bigskip
  \footnotesize

  \textsc{Department of Mathematics, University at Buffalo}
  \par\nopagebreak
  \textsc{Buffalo, NY 14260-2900, USA}  
  \par\nopagebreak
  \textit{E-mail address}: \texttt{achirvas@buffalo.edu}

}}

\maketitle

\begin{abstract}
  We prove that the cofree Hopf algebra on a cocommutative bialgebra is again cocommutative, and in fact coincides with the latter's largest Hopf subalgebra. This answers a question of Ardizzoni-Menini-Saracco and provides a proof for a statement to the same effect in the literature, whose original justification contains an apparent gap. The result relies on (and is in fact equivalent to) the observation that cocommutative bialgebra quotients of arbitrary Hopf algebras are automatically Hopf. This latter statement strengthens the analogue due to Newman, which assumes cocommutativity for the Hopf domain.

  One application of the above is the following quantum-rigidity result for Hopf algebras acted upon by a covariant quantum group: an action by one Hopf algebra on another which preserves the unital-algebra structure and the comultiplication factors through an action by a cocommutative Hopf algebra which moreover preserves the entire Hopf structure.
\end{abstract}


\vspace{.5cm}

\noindent{MSC 2020: 16T05; 16T15; 18A40; 16T20; 18M05; 18C15
  

}


\section*{Introduction}

The present note is primarily motivated by a question raised in passing in \cite{MR5016690} in the process of analyzing the functors left and right adjoint to the embedding of the category of Hopf algebras (over a field) in that of bialgebras. Given the automatic commutativity of the free Hopf algebra on a commutative bialgebra (\cite[Proposition 3.14]{MR5016690}, \cite[proof of Theorem 0.2]{2609.01541v2}), \cite[discussion preceding \S 3.3]{MR5016690} poses the natural dual question:

\begin{question}\label{qu:qc}
  Is the \emph{cofree Hopf algebra} \cite[Theorem 3.1(a)]{zbMATH05696924} on a cocommutative bialgebra again cocommutative?
\end{question}

That this is so is claimed in the earlier \cite[Remark 26]{porst_formal-2}, but that argument appears to contain a gap (see \cite[Remark 1.6(1) and Example 1.7]{2609.01541v2}). One of the main results provides a proof for that affirmative answer, in the following slightly sharper form.

\begin{theorem}\label{th:qc}
  \begin{enumerate}[(1),wide]
  \item\label{item:th:qc:cof.coc} For a cocommutative $\Bbbk$-bialgebra $B$, the cofree Hopf algebra $H(B)\xrightarrow{\iota_B} B$ is automatically cocommutative.

  \item\label{item:th:qc:cof.sub} Furthermore, the canonical map $\iota_B$ is the embedding of the largest Hopf subalgebra of $B$. 
  \end{enumerate}
\end{theorem}


\section{Hopf algebras maximal in and cofree on cocommutative bialgebras}\label{se:cofr.max}

Throughout, linear structures (algebras, coalgebras, bialgebras, etc.) are over a ground field $\Bbbk$ (occasionally satisfying additional constraints, to be specified where needed). 

\pf{th:qc}
\begin{th:qc}
  \begin{enumerate}[label={},wide]
  \item\textbf{\Cref{item:th:qc:cof.coc} $\Rightarrow$ \Cref{item:th:qc:cof.sub}} Recall first that a cocommutative bialgebra indeed \emph{has} a largest Hopf subalgebra (e.g. \cite[Remark 1.6(2)]{2609.01541v2}). Assuming \Cref{item:th:qc:cof.coc}, the cocommutative-domain map $\iota_B$ will factor through an embedding $B'\le B$ of a bialgebra quotient $H(B)\xrightarrowdbl{} B'$. The domain of that quotient being a cocommutative Hopf algebra, its bialgebra quotients are automatically Hopf \cite[Lemma 4.2]{MR374182}. $\iota_B$ is thus an embedding into the largest Hopf subalgebra of $B$, so also an isomorphism onto it by universality.
    
  \item\textbf{\Cref{item:th:qc:cof.coc}} Any bialgebra morphism $H\to B$ from a Hopf algebra will factor through a cocommutative quotient of $H$, automatically Hopf by \Cref{th:qq} (to which we relegate the bulk of the proof). That Hopf subalgebra of $B$ (cocommutative, since the latter is) thus has the requisite universality property and must be $H(B)$. 
  \end{enumerate}
\end{th:qc}

\Cref{th:qq} improves on the already-cited \cite[Lemma 4.2]{MR374182} (also \cite[Theorem 1(vi)]{nic}): not only are bialgebra quotients $H\xrightarrowdbl{} B$ of cocommutative Hopf algebras automatically Hopf, but this is so even if one only assumes $B$ (but not $H$) cocommutative.

\begin{theorem}\label{th:qq}
  A cocommutative bialgebra quotient of a Hopf algebra over a field is again Hopf. 
\end{theorem}
\begin{proof}
  Fix a quotient $H\xrightarrowdbl{\pi}B$ as in the statement. The invariance of a bi-ideal $\ker\pi\trianglelefteq H$ under the antipode can be checked after extending scalars to any convenient larger field, so there is no loss in assuming $\Bbbk$ algebraically closed. $B$ is then \emph{pointed} (i.e. its simple subcoalgebras are 1-dimensional) by \cite[Lemma 8.0.1(c)]{swe}, whence its being a Hopf algebra precisely when its grouplike elements are invertible \cite[Proposition 9.2.5]{swe}.

  Per standard Tannakian machinery (e.g. \cite[Theorem 2.4.2]{schau_tann}), the statement amounts to the claim that for any finite-dimensional $H$-comodule $V\in \cM^H_f$ the (automatically 1-dimensional) simple subquotients of $V$ regarded as a right $B$-comodule by corestriction along $\pi$ are \emph{rigid} \cite[Definition 2.10.11]{egno} in $\cM^B$. Indeed, every simple $B$-comodule is realizable as some such subquotient by \cite[Exercise (4) preceding Lemma 9.0.1]{swe}, say.
  
  By \cite[Theorem 8.0.5(c)]{swe} we have a decomposition
  \begin{equation*}
    V=\bigoplus_{g} V_g
    ,\quad
    g\in B\text{ grouplike}
    ,\quad
    V_g\text{ \emph{$g$-primary}},
  \end{equation*}
  the latter meaning that $V_g$ is a successive extension of copies of the 1-dimensional $D_g\in \cM^B$ attached to the grouplike $g$. Denoting by $V^*\in \cM^H$ the left dual, the evaluations $V^*\otimes V_g\to \Bbbk$ will be non-zero by the non-degeneracy of the duality pairing, and with them also $V^*\otimes D_g\to \Bbbk$. This makes $g$ left-invertible: for some 1-dimensional subquotient $D_h\le V^*$ in $\cM^B$ the non-zero map $D_h\otimes D_g\to \Bbbk$ translates to $hg=1$. Right invertibility (so invertibility period) follows:
  \begin{equation*}
    \begin{tikzpicture}[>=stealth,auto,baseline=(current  bounding  box.center)]
      \path[anchor=base] 
      (0,0) node (l) {$D_g$}
      +(3,.5) node (u) {$V\otimes V^*\otimes D_g$}
      +(6,0) node (r) {$V$}
      ;
      \draw[->] (l) to[bend left=6] node[pos=.5,auto] {$\scriptstyle \mathrm{coev}_V\otimes\id$} (u);
      \draw[->] (u) to[bend left=6] node[pos=.5,auto] {$\scriptstyle \id\otimes\mathrm{ev}$} (r);
      \draw[right hook->] (l) to[bend right=6] node[pos=.5,auto,swap] {$\scriptstyle \text{embedding}$} (r);
    \end{tikzpicture}
  \end{equation*}
  commutes, hence a non-vanishing component $\Bbbk\to V_g\otimes V^*$ for $\Bbbk\xrightarrow{\mathrm{coev}}V\otimes V^*$ and with it a non-vanishing map $\Bbbk \to D_g\otimes V^*$ to which the preceding argument applies.
\end{proof}

We pause also to note the following consequence on (covariant) \emph{quantum automorphism groups} of Hopf algebras. The dual counterpart \cite[Theorem 0.2(2)]{2609.01541v2}, pertaining to \emph{contra}variant automorphism groups (i.e. coacting Hopf algebras), in turn builds on the dual to \Cref{th:qc}\Cref{item:th:qc:cof.coc}: the automatic commutativity of the Hopf envelope of a commutative bialgebra. 

\begin{corollary}\label{cor:univ.coact.fact.cocom}
  An action $B\otimes K\to K$ of a Hopf algebra on another, $(K,\eta,\mu,\varepsilon,\Delta,S)$, preserving the unit, multiplication and comultiplication factors through an action by a cocommutative Hopf quotient $B\xrightarrowdbl{}B'$ $B'\otimes K\to K$ preserving the entire Hopf structure.
\end{corollary}
\begin{proof}
  It so factors through a cocommutative \emph{bialgebra} quotient by \cite[Theorem 0.2(3)]{2609.01541v2}, automatically Hopf by \Cref{th:qq}. 
\end{proof}

\begin{remarks}\label{res:rev.impl}
  \begin{enumerate}[(1),wide]
  \item \Cref{th:qc} appeals to \Cref{th:qq}; the logical dependence can be reversed: taking the former for granted instead, I claim the latter follows. Indeed, a cocommutative bialgebra quotient $H\xrightarrowdbl{}B$ of a Hopf algebra will factor through the cofree Hopf algebra $H(B)\xrightarrow{\iota_B}B$, assumed cocommutative and an embedding; surjectivity then ensures $\iota_B$ is an isomorphism.

  \item One effect of the above-noted equivalence between \Cref{th:qc} and \Cref{th:qq} is to bring the former in scope for scalar extension: the latter proof is unproblematic in that regard, whereas cofree constructions tend to raise difficulties \cite[Theorem 2.9]{2407.20154v3} under non-algebraic field extensions. 
  \end{enumerate}
\end{remarks}


\addcontentsline{toc}{section}{References}

\def\polhk#1{\setbox0=\hbox{#1}{\ooalign{\hidewidth
  \lower1.5ex\hbox{`}\hidewidth\crcr\unhbox0}}}


\Addresses

\end{document}